\documentclass[10pt,a4paper]{article}
\usepackage[utf8]{inputenc}
\usepackage[T1]{fontenc}

\usepackage{amssymb,amsmath,amstext,euscript,tabularx,epsfig,amsthm, amscd, amsfonts, graphicx,float,pgf,tikz,amsfonts}  
\usepackage{here,enumerate,color,indentfirst,fontenc, verbatim}
\usepackage{delarray,multicol,hhline,alltt}
\usepackage{makeidx}
\usepackage{amsmath}
\usepackage{bbold}
\usepackage{mathtools,tikz-cd}
\usepackage{tikz}
\usepackage{amsfonts}
\usepackage{amssymb}
\usepackage{setspace}
\usepackage{graphicx}
\usepackage{color}
\usepackage{hyperref}
\usepackage{fancybox}
\usepackage{latexsym,wasysym,mathrsfs}
\usepackage[square,numbers,sort&compress]{natbib}
\usepackage{lipsum}
\usepackage{fancyhdr}
\usepackage{time}
\usepackage{color}
\usepackage[Glenn]{fncychap}
\usepackage{afterpage}
\usepackage{xcolor}
\usepackage{blindtext}
\author{Chaïmae Mezzat}

\usepackage{authblk}
\newtheorem{defi}{Definition}[section]
\newtheorem{ex}{Example}

\newtheorem{rmq}{Remark}[section]
\newtheorem{thm}{Theorem}[section]
\newtheorem{cor}{Corollary}[section]
\newtheorem{pro}{Proposition}[section]

\newtheorem{lem}{Lemma}[section]

\newcommand{\C}{\mathbb{C}}
\newcommand{\R}{\mathbb{R}}

\newcommand{\N}{\mathbb{N}}
\newcommand{\ind}{\mathbb{1}}

\newcommand{\B}{\mathcal{B}}
\newcommand{\A}{\mathcal{A}}
\newcommand{\h}{\mathcal{H}}
\newcommand{\z}{\mathcal{Z}}

\newcommand{\el}{\mathcal{L}}
\newcommand{\cnt}{\mathcal{C}}
\newcommand{\X}{\mathcal{X}}
\newcommand{\ran}{\mathcal{R}}

\newcommand{\keywords}[1]{\textbf{\textit{Keywords---}} #1}

\newcommand{\subjclass}[2][2010]{
  Subject classification\footnotetext{ \emph{2020 Mathematics subject classification.}  #2 }}
\definecolor{Gris}{rgb}{0.45,0.48,0.45}
\definecolor{lightgray}{gray}{0.9}
\colorlet{myblue}{blue!45!gray}

\begin{document}

 \thispagestyle{plain}
 \begin{titlepage}

  \vspace{2cm}
\hspace{0.5cm}  \begin{minipage}{16cm}

    \vspace{2cm}
    \begin{center}
    \begin{large}
  \Large{ \textbf{\textit{ \hspace{-1cm}$K$-$b$-frames for Hilbert spaces and the $b$-adjoint operator}}}
    \end{large}
    \end{center}
    \begin{center}  
    Chaimae MEZZAT$^{(1)}$, Samir KABBAJ$^{(2)}$ and Abdelkarim BOUROUIHIA$^3$
   \newline $~^{(1)}$ chaimae.mezzat7@gmail.com. Department of Mathematics, University of Ibn Tofail, B.P. 133, Kenitra, Morocco.
  \newline  $~^{(2)}$ samkabbaj@yahoo.fr. Department of Mathematics, University of Ibn Tofail, B.P. 133, Kenitra, Morocco.
\newline $~^{(3)}$ ab1221@nova.edu. Department of Mathematics, Nova Southeastern University, 3301 College Avenue, Fort Lauderdale, Florida, United States of America.
   \end{center}
 \end{minipage}
 
 \begin{minipage}[c]{\linewidth}
 \vspace{1cm}
 \begin{center}
 Laboratory of Partial Differential Equations, Algebra and Spectram Geometry.\\
 Faculty of sciences, Ibn Tofail University\\
 Kénitra, Morocco.
 \end{center}
 \end{minipage}

 \begin{minipage}[c]{\linewidth}
  \vspace{1cm}
  \begin{abstract}
  In this paper, we will generelize $b$-frames; a new concept of frames for Hilbert spaces, by $K$-$b$-frames. The idea is to take a sequence from a Banach space and see how it can be a frame for a Hilbert space. Instead of the scalar product we will use a new product called the $b$-dual product and it is constructed via a bilinear mapping. We will introduce new results about this product, about $b$-frames, and about $K$-$b$-frames, and we will also give some examples of both $b$-frames and $K$-$b$-frames that have never been given before. We will give the expression of the reconstruction formula of the elements of the Hilbert space. We will as well study the stability and preservation of both $b$-frames and $K$-$b$-frames; and to do so, we will give the equivalent of the adjoint operator according to the $b$-dual product.
  \end{abstract}
 \end{minipage}
 
 \begin{minipage}[c]{\linewidth}
 \vspace{1cm}
 \end{minipage}

 \begin{minipage}[c]{\linewidth}
\vspace{1.5cm}
\end{minipage}
\end{titlepage} 

\newpage
\tableofcontents
 
 \newpage
\begin{center}
\Large{\textbf{Introduction}}
\end{center}
The history of Frames Theory has begun in 1952, when Duffin and Schaeffer introduced it to solve some nonharmonic Fourier series problems \cite{Duffin}. Frames were a key that solved several problems in many fields, such as physics, natural sciences, and engineering. It is applied in signal and image processing and treatment, coding, system modeling and many other interesting fields (see \cite{grossmann}, \cite{wavelets} and \cite{balan})\\
 It is also known that D.Gabor has already introduced "Les Gaborettes" in 1946 which are a particular case of frames, but they were introduced for a physics purpose.\\
 In 1986, I.Daubechies, A.Grossmann, and Y.Meyer revolutionized the Theory of Frames by giving redundancy properties. The interest was to find a reconstruction formula using frames to feed the need of many problems in numerous fields.\\
 In 2012, L.Gavruta introduced a more generale case of frames, which is $K-$Frames and proved its existence in \cite{Gavruta} and was able to extend many existing results to it and gave some new interesting results in \cite{Gavruta2}. For more results about $K$-frames see (\cite{kabros} and \cite{kabbross}).\\
 In order to generalize frames for Hilbert spaces, M.Ismailov, F.Guliyeva, and Y.Nasibov came up with the concept of $b$-frames. It is a concept where we are able to generate Hilbert frames by a bilinear mapping $b: \h_1\times \B \rightarrow \h_2$, such that $\h_1,\h_2$ are two Hilbert spaces and $\B$ is a Banach space, and via this bilinear mapping we construct a new product that we will call the $b-$dual product to define a new frame for a Hilbert space, but this time not from the Hilbert space itself; it is from the Banach space $\B$.\\
 Moreover M.Ismailov, F.Guliyeva, and Y.Nasibov extended some well known and important results of frame theory existing in the classical case to this more general one (see \cite{b-frame}). Except that the definition was incomplete, and there was no example given. Also the stability and preserving were not studied.\\ 
In this paper we will give a detailed definition of $b$-frames and generalize this definition by giving the definition of $K$-$b$-frames, and we will give some examples of both of them. In addition we will introduce some new results about both $b$-frames and $K$-$b$-frames, we will explain in details the $b-$dual product used to define this new concept of frames. We will give the reconstruction formula for both $b$-frames and $K$-$b$-frame. We will introduce the $b$-adjoint operator to be able to manipulate the $b$-dual product in order to study the stability and preservation of both $b$-frames and $K$-$b$-frames.

\newpage

\textbf{\textit{\underline{Notations:}}}
\begin{enumerate}
\item By $\cnt(I)$ we denote the set of continuous functions mapping from an interval $I\subset \C$ to $\C$.
\item By $\cnt_n(I)$ we denote the set of piecewise-continuous functions from $I\subset \C$ to $\C$.
\item If $\h$ and $\z$ are two normed spaces, then $\el(\h,\z)$ denotes the set of bounded linear operators from $\h$ to $\z$.
\item If $T\in\el(\h,\z)$, then $T^*\in\el(\z,\h)$ is the adjoint operator of $T$ and it verifies $\langle Tx,y\rangle_{\z}=\langle x,T^*y\rangle_{\h}$, for all $x\in\h$ and $y\in\z$.
\item $Id_{\z}$ denotes the identity operator of $\z$.
\item $\ran(T)$ denotes the range of $T$.
\item $\ker(T)$ denotes the kernel of $T$.
\item Let $\h$ be a Hilbert space and let $T\in\el(\h)$, such that $\ran(T)$ is closed. $T^{\dagger}$ will denote the pseudo-inverse (or the moore-penrose inverse) of $T$, verifying $TT^{\dagger}h=h,~~\forall h\in\ran(T)$.
\end{enumerate}

\section{PRELIMINARIES}

\begin{thm}\label{douglas}\cite{Douglas}.
Let $\h,\h_1,\h_2$ be Hilbert spaces, and let $S \in \el(\h_1,\h)$, and $T \in \el(\h_2,\h)$. The following statements are equivalent:
\begin{enumerate}
\item[i)] $\ran(S)\subset \ran(T)$;
\item[ii)] $SS^* \leq \lambda^2 TT^*,$ for some $\lambda>0$;
\item[iii)] there exists $U\in \el(\h_1,\h_2)$ such that $S=TU$.
\end{enumerate}
\end{thm}

\begin{thm}\label{prop de T}\cite{b-frame}.
Let $T\in \el(\h,\z)$, where $\h$ and $\z$ are two Hilbert spaces. Then,
\begin{enumerate}
\item[i)] $T^*\in\el(\z,\h)$, and $\Vert T^*\Vert=\Vert T\Vert$;
\item[ii)] $\ran(T)$ is closed if and only if $\ran(T^*)$ is closed;
 \item[iii)] $T$ is surjective if and only if $\exists c>0$, such that $\Vert T^*z\Vert_{\z} \geq c\Vert z\Vert_{\z}$, $\forall z\in \z$;
\end{enumerate}
\end{thm}

Let $\h, \z$ be two hilbert spaces and let $\left\langle . , .\right\rangle_{\h}$, $\left\langle .,.\right\rangle_{\z}$ be their corresponding scalar products respectively. We denote by $\Vert .\Vert_{\h}$ (resp. $\Vert .\Vert_{\z}$) the norm of $\h$ (resp. $\z$). 
Let $\B$ be a Banach space with norm $\Vert .\Vert$.\\
Now consider the bilinear mapping $\flat : \h \times \B \longrightarrow \z$, satisfying the condition
\begin{equation}\label{majo}
\exists M>0: ~ \Vert \flat(h,x)\Vert_{\z} \leqslant M \Vert h\Vert_{\h} \Vert x \Vert, \hspace{1cm} \forall h \in \h,~ x \in \B.
\end{equation}
Note that the inequation (\ref{majo}) means that the mapping $\flat$ is bounded and continuous.\\
Fix $z \in \z$ and $x \in \B$, and consider the linear functional $\phi_{z,x}: \h \longrightarrow \C$, defined by\\ $\phi_{z,x}(h)=\left\langle \flat(h,x),z\right\rangle_{\z}$. It is clear that $\phi_{z,x}$ is linear, also by (\ref{majo}) we have, for all $h \in \h$
\begin{align*}
\vert\phi_{z,x}(h)\vert &= \vert \left\langle \flat(h,x),z\right\rangle_{\z}\vert \leqslant \Vert\flat(h,x)\Vert_{\z}\Vert z\Vert_{\z}\\
& \leqslant M \Vert h\Vert_{\h}\Vert x\Vert \vert z\Vert_{\z}.
\end{align*}
Hence, $\Vert \phi_{z,x}\Vert \leqslant M \Vert x\Vert \Vert z\Vert_{\z}$. By the Riesz representation theorem, it follows that there exists a unique element $v\in \h$ such that $\phi_{z,x}(h)=\left\langle h,v\right\rangle_{\h}$, and this element will be called $\flat$-dual product of $z$ and $x$ and will be denoted $\left\langle z/x \right\rangle$. In what follows we will study the properties of this product.\\
Set
\begin{align*}
\xi: \z\times\B &\rightarrow \h\\
(z,x) & \mapsto \langle z/x\rangle
\end{align*}
\begin{pro}\label{dualprod} The mapping $\xi$ has the following properties:
\begin{enumerate}
\item[1)] $\xi$ is sesquilinear.
\item[2)] $\xi$ is bounded and continuous.
\item[3)] Suppose that $\overline{\ran}(\flat)=\z$. If for every $x\in\B$, $\langle z/x\rangle=0$ then $z=0$. We say that $\xi$ is non-degenerated with respect to $z$.
\item[4)] Suppose that there exists $m>0$ such that $$ m\Vert h\Vert_{\h}\Vert x\Vert \leqslant \Vert\flat(h,x)\Vert_{\z}\Vert z\Vert_{\z},~~ \forall x\in \B,\forall h\in\h.$$ If for every $z\in\z$, $\langle z/x\rangle=0,$ then $x=0$. We say that $\xi$  is non-degenerated with respect to $x$.
\end{enumerate}
\end{pro}
\begin{proof}
\begin{enumerate}
\item[1)] 
 Let $z\in\z$, $x_1,x_2\in \B$, and $\lambda\in\C$,
\begin{align*}
\langle h, \langle z/\lambda x_1+x_2\rangle\rangle_{\h}=\langle \flat(h,\lambda x_1+x_2),z\rangle_{\z}&= \langle\lambda\flat(h,x_1),z\rangle_{\z}+\langle \flat(h,x_2),z\rangle_{\z}\\&= \lambda\langle h, \langle z/x_1\rangle\rangle_{\h}+\langle h, \langle z/x_2\rangle\rangle_{\h}\\
&=\langle h,\bar{\lambda} \langle z/x_1\rangle\rangle_{\h}+\langle h, \langle z/x_2\rangle\rangle_{\h}.
\end{align*}
Then $\xi$ is anti-linear with respect to the second variable. In the same way, we prove that $\xi$ is linear with respect to the first variable.
\item[2)] $\Vert \langle z/x\rangle\Vert_{\h}=\Vert \phi_{z,x}\Vert\leqslant M\Vert x\Vert \Vert z\Vert_{\z}, ~\forall (z,x)\in \z\times\B$, then $\xi$ is bounded and continuous for the product topology of $\z\times\B$.
\item[3)] Suppose that $\overline{\ran}(\flat)=\z$, and that for all $x\in\B$ $\langle z/x\rangle=0$, this implies that 
, $$\langle h,\langle z/x\rangle\rangle_{\h}=0,~~~\forall h\in\h,~\forall x\in\B.$$
$\Rightarrow$
$$\langle \flat(h,x),z\rangle_{\z}=0,~~~\forall h\in\h,~\forall x\in\B.$$
Hence $$z=0.$$
\item[4)] \begin{align*}
\langle z/x\rangle =0, \forall z\in\z &\Rightarrow \langle h,\langle z/x\rangle\rangle_{\h}=0, ~~\forall h\in\h,\forall z\in\z\\
&\Rightarrow \langle \flat(h,x),z\rangle_{\z}=0,~~\forall h\in\h,\forall z\in\z\\
&\Rightarrow \flat(h,x)=0,~~\forall h\in\h,\\
&\Rightarrow m\Vert h\Vert_{\h}\Vert x\Vert=0,~~\forall h\in\h\\
&\Rightarrow x=0.
\end{align*} 
 \end{enumerate}
 \end{proof}
Now let $\{x_i\}_{i\in \N}\subset \B$ be a vector sequence, we will first give some necessary definitions :
\begin{defi}\cite{b-frame}.
$\{x_i\}_{i\in \N}$ is said to be $\flat$-orthonormal in $\z$ if
\begin{equation}
\langle \flat(h,x_i)/x_j\rangle=\delta_i^jh, ~~~\forall h\in\h,~i,j\in \N.
\end{equation}
\end{defi}

\begin{defi}\cite{b-frame}.
$\{x_i\}_{i\in \N}$ is said to be a $\flat$-basis for $\z$ if for every $z \in \z$, there exists a unique sequence $\{h_i\}_{i\in\N}$ in $\h$ such that:
$$ z=\sum_{i=1}^{\infty} \flat(h_i,x_i).$$
if in addition $\{x_i\}_{i\in \N}$ is $\flat$-orthonormal, then $\{x_i\}_{i\in \N}$ is called $\flat$-orthonormal basis in $\z$.
\end{defi}


\begin{defi}\label{complisdegen}\cite{b-frame}.
Let $\{x_i\}_{i\in\N}\subset \B$. $\{x_i\}_{i\in\N}$ is said to be $\flat$-complete in $\z$ if and only if
 \begin{equation}\label{nondeg}
\langle z/x_i\rangle=0,~\forall i\in\N \Rightarrow z=0.\end{equation}
\end{defi}
 
\begin{rmq}
 If in addition to (\ref{majo}) the bilinear mapping $\flat$ verifies 
\begin{equation}
\exists m>0;~~ m\Vert h\Vert_{\h}\Vert x\Vert \leqslant \Vert \flat(h,x)\Vert_{\z}, ~~~\forall h\in\h, \forall x\in\B.
\end{equation}
and $\overline{Span}^{\B}\{x_i/i\in\N\}=\B$. Then $\{x_i\}_{i\in\N}$ is $\flat$-complete in $\z$ if and only if $\overline{\ran(\flat)}^{\z}=\z$. 
\end{rmq}

\begin{pro}\cite{b-frame}.
Let $\{x_i\}_{i\in\N}\subset\B$, we have the following equivalence:
$$\left[\{x_i\}_{i\in\N} \text{ is a } \flat\text{-orthonormal basis in }\z\right]\Leftrightarrow \left[\{x_i\}_{i\in\N}\text{ is }\flat\text{-complete in }\z\right].$$
\end{pro}

\begin{ex}
Let $-\infty<a<b$ and let $\omega:[a,b[\rightarrow \R^*_+$ be a continuous weight,\\(i.e $\forall n\in\N; \int_a^b \vert x\vert^n\omega (x)dx<+\infty$).\\
We consider the Hilbert space 
$$H_{\omega}=\left\lbrace h\in\cnt([a,b[)/\int_a^b\vert h(t)\vert^2\omega(t)dt<\infty\right\rbrace.$$
And the space
$$H'_{\omega}=\left\lbrace h\in\cnt_m([a,b[)/\int_a^b\vert h(t)\vert^2\omega(t)dt<\infty\right\rbrace.$$
Note that $\overline{H'_{\omega}}^{\Vert .\Vert_{\infty}}=H_{\omega}$ because $\overline{\C_m([a,b[)}=\C([a,b[)$ everywhere.
Provide $H'_{\omega}$ with the scalar product and norm such as
$$\langle h,g\rangle_{\omega}=\int_a^b h(t)\bar{g(t)}\omega(t)dt,~ and ~~~  \Vert h\Vert_{\omega}=\left(\int_a^b \vert h(t)\vert^2\omega(t)dt\right)^{\frac{1}{2}}.$$
Let 
\begin{align*}
b:H'_{\omega}\times L^{\infty}([a,b[)\cap\cnt_m([a,b[)&\rightarrow \overline{L^2([a,b[)\cap\cnt_m([a,b[)}^2\\
(h,g)&\mapsto hg\sqrt{\omega}.
\end{align*}
We have 
$$\Vert b(h,g)\Vert_2=\left( \int_a^b\vert h(t)\vert^2\vert g(t)\vert^2\omega(t)dt\right)^{\frac{1}{2}}
\leq \sup_{t\in[a,b[}\vert g(t)\vert\left(\int_a^b\vert h(t)\vert^2\omega(t)dt\right)^{\frac{1}{2}}
\leq \Vert h\Vert_{\omega}\Vert g\Vert_{\infty}.$$
And 
$$\langle  b(h,g),f\rangle_2=\int_a^b b(h,g)(t)\bar{f(t)}dt
=\int_a^b hg\sqrt{\omega}(t)\bar{f(t)}d(t)
=\int_a^b h\overline{\bar{g}(t)f(t)}\omega^{-\frac{1}{2}}(t)\omega(t)dt= \langle h, \bar{g}f\omega^{-\frac{1}{2}}\rangle_{\omega}.$$
Hence we have $\langle f/g\rangle(t)=\bar{g}(t)f(t)\omega^{-\frac{1}{2}}(t)$, $\forall t\in [a,b[$.\\
Let $\Delta=\{a=a_0,..,a_k=a+k\frac{b-a}{m},..,a_m=b\}$ be a subdivision of $[a,b[$, and consider the $b$-orthogonal sequence $\{g_i\}_{i\in\N}$, such that $g_i(t)=\omega^{-\frac{1}{4}}(t)\ind_{[a_i,a_{i+1}[}(t)$; in fact we have $\langle b(h,g_i/g_j\rangle=\delta_i^j h$. In addition, if we suppose that $\langle f/g_i\rangle=0$ for all $i\in\N$, then we have $\omega^{-\frac{3}{4}}(t)\ind_{[a_i,a_{i+1}[}(t)f(t)=0$, which means that $f\equiv0$ on each subdivision $[a_i,a_{i+1}[$ of $[a,b[$, hence $f\equiv0$ on $[a,b[$, so $\{g_i\}_{i\in\N}$ is $b$-complete in $\overline{L^2([a,b[)\cap\cnt_m([a,b[)}^2$.
\end{ex}

\begin{ex}\label{exemple b-orth}(In finite dimension)\\
Let $\h=\R^4$ with basis $\{e_i\}_{1\leq i\leq 4}$, $\z=\R^3$ with basis $\{v_1,v_2,v_3\}$, and $\B=\R^2$ with basis $\{f_1,f_2\}$. Let $b:\h\times\B\rightarrow\z$ be as follows:
\begin{align*}
b(e_1,f_1)&=u_1,~~~~ b(e_2,f_1)=u_3,~~~b(e_3,f_1)=0,~~~b(e_4,f_1)=0
\\b(e_1,f_2)&=u_2,~~~~b(e_2,f_2)=0,~~~~ b(e_3,f_2)=0,~~~                                                                                                                                                                                                                                                                                                                                                                                                                                                                         ; b(e_4,f_2)=0.
\end{align*}
for each $h=\sum_{i=1}^4h_ie_i\in\h$, and $x=x_1f_1+x_2f_2\in\B$ we have, $b(h,x)=h_1x_1u_1+h_1x_2u_2+h_2x_1u_3,$
and $\Vert b(h,x)\Vert_{\z}^2=\Vert h_1x_1u_1+h_1x_2u_2+h_2x_1u_3\Vert_{\z}^2=h_1^2x_1^2+h_1^2x_2^2+h_2^2x_1^2=(h_1^2+h_2^2)(x_1^2+x_2^2)-(h_2x_2)^2$
So
$$\Vert b(h,x)\Vert_{\z}\leq \beta \Vert h\Vert_{\h}\Vert x\Vert.$$
where $\beta=1$.
By a simple calculus we obtain  the $b$-dual product for every $z\in\z$, and every $x\in\B$, $\langle z/x\rangle= (x_1z_1+x_2z_2)e_1+x_1z_3e_2.$\\
Chose $\{X_1,X_2\}\subset \B$ such that $X_1=f_1-f_2$ and $X_2=f_1+f_2$, we have
$$\left \{
\begin{array}{rcl}
 \langle z/X_1\rangle&=&0 \\
\langle z/X_2\rangle &=&0
\end{array}
\right.
\Rightarrow
\left \{
\begin{array}{rcl}
 z_1-z_2&=&0 \\
 z_1+z_2&=&0\\
z_3&=&0
\end{array}
\right.
\Rightarrow
z=0
$$
Then $\{X_1,X_2\}$ is a $\flat$-orthonormal basis in $\z$.
\end{ex}


\section{Main results}

\subsection{$\flat$-frames and some examples}
In this section we will give a more detailed definition of $b$-frames, that were introduced in \cite{b-frame}, and we will give some examples that have never been given in the litterature to see how this concept works.
\begin{defi}
The sequence $\{x_i\}_{i\in \N} \subset \B$ is called a $\flat$-frame for $\z$ if there exist constants\\ $0<A\leq B <\infty$ such that
 \begin{equation}\label{defb-frame}
 A \Vert z\Vert^2_{\z} \leqslant \sum_{i=1}^{\infty}\Vert \left\langle z/ x_i\right\rangle\Vert^2_{\h} \leqslant B\Vert z\Vert^2_{\z}, ~~~\forall z \in \z.
 \end{equation}
\end{defi}
Constants $A$ and $B$ are called the bounds of $\flat$-frame. When the right side of (\ref{defb-frame}) is verified, we say that the sequence $\{x_i\}_{i\in\N}$ is $\flat$-Besselian in $\z$ with bound $B$.\\
If $A=B$, then $\{x_i\}_{i\in\N}$ is a tight $\flat$-frame for $\z$.\\
If $A=B=1$, then $\{x_i\}_{i\in\N}$ is Parseval $\flat$-frame for $\z$.

\begin{ex}\label{expl}
Choose $\h=\R^3$, $\B=\R^2$, and $\z=\R^4$ such that $(e_1,e_2,e_3),(f_1,f_2), and (u_1,u_2,u_3,u_4)$ are their canonical bases respectively. Define the bilinear mapping $b:\h\times\B\rightarrow \z$ as follows
\begin{align*}
b(e_1,f_1)&=u_1,~~~~ b(e_2,f_1)=u_3,~~~~b(e_3,f_1)=u_1,
\\b(e_1,f_2)&=u_2,~~~~b(e_2,f_2)=u_4,~~~~ b(e_3,f_2)=u_2.
\end{align*}
for each $h=\sum_{i=1}^3h_ie_i,$ and $x=\sum_{i=1}^2x_if_i$, we obtain
$$b(h,x)=(h_1+h_3)x_1u_1+(h_1+h_3)x_2u_2+h_2x_1u_3+h_2x_2u_4,$$
hence
$$\Vert b(h,x)\Vert^2_{\z}=(h_1^2+h_2^2+h_3^2)(x_1^2+x_2^2)+2h_1h_3(x_1^2+x_2^2)\leqslant 2\Vert h\Vert^2\Vert x\Vert^2.$$

for $z\in\z$ we have 
$$\langle b(h,x),z\rangle=\langle h,\langle z/x\rangle\rangle
=h_1(x_1z_1+x_2z_2)+h_2(x_1z_3+x_2z_4)+h_3(x_1z_1+x_2z_2).
$$
$~~~~\Rightarrow~~~~ \langle z/x\rangle=(x_1z_1+x_2z_2)e_1+(x_1z_3+x_2z_4)e_2+(x_1z_1+x_2z_2)e_3.$\\
Choose $y_1=f_1+f_2,~~ y_2=f_1-f_2$ in $\B$. We have 
$$\Vert \langle z/y_1\rangle\Vert^2_{\h}=2 z_1^2+2 z_2^2+ z_3^2+ z_4^2+4 \vert z_1z_2\vert+2 \vert z_3z_4\vert.$$
and
$$\Vert \langle z/y_2\rangle\Vert^2_{\h}=2 z_1^2+2 z_2^2+ z_3^2+ z_4^2-4 \vert z_1z_2\vert-2 \vert z_3z_4\vert.$$
we obtain, $\forall z\in \z$:
$$\Vert z\Vert_{\z}^2\leqslant\Vert \langle z/y_1\rangle\Vert^2_{\h}+\Vert \langle z/y_2\rangle\Vert^2_{\h}=4 z_1^2+4 z_2^2+2 z_3^2+2 z_4^2\leqslant 4\Vert z\Vert_{\z}^2.$$
Then $\{y_1,y_2\}$ is a $b$-frame for $\z$ with bounds $A=1$ and $B=4$.
\end{ex}

\begin{ex}
Consider the Hilbert spaces $\h=\z=L^2(\C)$ and the Banach space $\B=L^{\infty}(\C)$, and consider the bilinear mapping 
\vspace{-1cm}
\begin{align*}
b: \h\times \B &\rightarrow \z\\
(h,g) &\mapsto hg
\end{align*}
$b$ is well defined and verifies (\ref{majo}), in fact
$$\Vert b(h,g)\Vert_{\z}=\Vert hg\Vert_2= \sqrt{\int_{\C}\vert h(t)g(t)\vert^2 dt}
\leqslant \Vert g\Vert_{\infty}\sqrt{\int_{\C}\vert h(t)\vert^2 dt}
\leqslant \Vert h\Vert_2~ \Vert g\Vert_{\infty}<\infty.$$
now let $f\in\z$, we have
$$\langle b(h,g),f\rangle_{\z}= \int_{\C} h(t)g(t)\bar{f}(t) dt=\int_{\C} h(t)\overline{\bar{g}(t)f(t)} dt
=\langle h,\langle f/g\rangle\rangle_{\h}.$$
Hence $\langle f/g\rangle=\bar{g}f,~~\forall f\in\z,\forall g\in \B$.
Let $\{g_i\}_{i\in\N}\subset\B$, such that $g_i(t)=\frac{1}{i+1}$, $\forall i\in\N, \forall t\in\C$.\\
 One has
$$\Vert \langle f/g_i\rangle\Vert^2_{\h}= \Vert g_i\bar{f}\Vert^2_2= \int_{\C}\left\vert \frac{1}{i+1}\bar{f}(t)\right\vert^2dt=\frac{1}{(i+1)^2}\int_{\C}\vert f(t)\vert^2dt=\frac{1}{(i+1)^2}\Vert f\Vert^2_2,$$
then, $$\sum_{i\in\N}\Vert \langle f/g_i\rangle\Vert^2_{\h}=\frac{\pi^2}{6}\Vert f\Vert^2_2.$$
Hence $\{g_i\}_{i\in\N}$ is a tight $b$-frame for $\z$ with bound $A=B=\frac{\pi^2}{6}$. 
\end{ex}
 
\subsection{$K$-$\flat$-frames and some examples}
In this section we generalize the concept of $b$-frames by the same idea that Gavruta generalized frames for Hilbert spaces by $K$-frames for Hilbert spaces. Recall first the definition given by Gavruta.
 \begin{defi}\cite{Gavruta}
Let $K\in \el(\z)$ and let $K^*$ be its adjoint. A sequence $\{z_i\}_{i\in\N}\subset\z$ is a $K$-frame for $\z$ if and only if there exist constants $0<A<B<\infty$, such that
\begin{equation}\label{frame}
A\Vert K^*z\Vert_{\z}^2 \leqslant \sum_{i=1}^{\infty} \vert\langle z,z_i\rangle_{\z}\vert^2\leqslant B \Vert z\Vert_{\z}^2,~~ \forall z\in\z.
\end{equation}
  constants $A$ and $B$ are called lower and upper frame bound respectively.
\end{defi}
If $K\equiv Id_{\z}$, then $\{f_i\}_{i\in\N}$ is a frame for $\z$.
Now let $K\in\el(\z)$, and let $K^*$ be its adjoint, we define a $K$-$\flat$-frame as follows:
 \begin{defi}\label{k-b-frame}
 The sequence $\{x_i\}_{i\in \N} \subset \B$ is called a $K$-$\flat$-frame for $\z$ if there exist constants $0<A\leq B <\infty$ such that
 \begin{equation}\label{k-b-frame}
 A \Vert K^*z\Vert^2_{\z} \leqslant \sum_{i=1}^{\infty}\Vert \left\langle z/ x_i\right\rangle\Vert^2_{\h} \leqslant B\Vert z\Vert^2_{\z}, ~~~\forall z \in \z.
 \end{equation}
 \end{defi}
 The constant $A$ (resp, $B$) is called left (resp, right) bound of the $K$-$\flat$-frame.
When only the right inequation of (\ref{k-b-frame}) is fulfilled, then the sequence $\{x_i\}_{i\in N}$ is called $K$-$\flat$-Besselian in $\z$ with bound $B$.\\
If $A=B$, then $\{x_i\}_{i\in\N}$ is a tight $K$-$\flat$-frame for $\z$.\\
If $A=B=1$, then $\{x_i\}_{i\in\N}$ is Parseval $K$-$\flat$-frame for $\z$.
\begin{rmq}
Note that a $\flat$-frame is a $K$-$\flat$-frame with $K\equiv Id_{\z}$.
\end{rmq}

\begin{ex}
Let $\h_1,\h_2, \z_1,\z_2$ be Hilbert spaces. Let $\z=\z_1\oplus \z_2$, $\h=\h_1\oplus \h_2$ and let $\B=\el(\h_1,\z_1)\oplus \el(\h_2,\z_2)$ be a Banach space.\\
Consider the bilinear mapping $\flat : \h\times \B \longrightarrow \z$ such that $ \flat ((h_1\oplus h_2),(t_1\oplus t_2)) =t_1(h_1)\oplus t_2(h_2)$.
Let $z_1\oplus z_2\in\z$, we have 
$$\langle \flat((h_1\oplus h_2),(t_1\oplus t_2)),z_1\oplus z_2\rangle_{\z}= \langle t_1(h_1)\oplus t_2(h_2),z_1 \oplus z_2\rangle_{\z}
=\langle h_1\oplus h_2,t_1^*(z_1)\oplus t_2^*(z_2)\rangle_{\h}.$$
We claim that 
$$\langle z_1\oplus z_2/t_1\oplus t_2\rangle = t_1^*(z_1)\oplus t^*_2(z_2)=\langle z_1/t_1\rangle\oplus\langle z_2/t_2\rangle.$$
 Now suppose that $(t_1^i)_{i\in \N}$ is a $\flat$-frame for $\z_1$ , i.e, there exist constants $0<A_1\leq B_1<\infty$, such that
\begin{equation}\label{g-frame1}
 A_1\Vert z_1\Vert^2_{\z_1} \leqslant \sum_{i\in\N} \Vert t_1^{i*}(z_1)\Vert_{\h_1}^2\leqslant B_1\Vert z_1\Vert^2_{\z_1}, ~~ \forall z_1 \in \z_1.
\end{equation}
If we suppose that $(t_2^i)_{i\in \N} \approxeq \circleddash$ (i.e. all the operators $t_2^i$ are zero), then $(t_1^i \oplus t_2^i)_{i\in \N}$ isn't a $\flat$-frame for $\z$, but if we set 
$K: \z \longrightarrow \z$ such that, for every $z=(z_1\oplus z_2)\in \z$, $K(z_1\oplus z_2)=z_1$. Then we have $\Vert K^*(z)\Vert^2_{\z}=\Vert z_1\Vert^2_{\z_1}$, because $K$ is self adjoint.\\
So, for each $z_1\in\z_1, z_2\in\z_2$ we have,
\begin{align*}
A_1\Vert K^*(z_1\oplus z_2)\Vert^2_{\z} \leqslant \sum_{i\in\N} \Vert t_1^{i*}(z_1)\Vert_\z^2= \sum_{i\in\N} \Vert t_1^{i*}\oplus t_2^{i*}(z_1\oplus z_2)\Vert_\z^2&=\sum_{i\in\N} \Vert \langle z_1\oplus z_2/t_1\oplus t_2\rangle\Vert_\h^2\\
& \leqslant B_1\Vert z_1\Vert^2_{\z_1} \leq B_1\Vert z_1\oplus z_2\Vert^2_{\z}.
\end{align*}
Which shows that $\{t_1^i\oplus t_2^i\}_{i\in\N}$ is a $K$-$\flat$-frame for $\z$ with bounds $A_1$ and $B_1$.
\end{ex}

\begin{ex}
Choose $\h=\R^3$, $\B=\R^2$, and $\z=\R^4$ such that $(e_1,e_2,e_3),(f_1,f_2), and (u_1,u_2,u_3,u_4)$ are their canonical bases respectively. Define the bilinear form $b:\h\times\B\rightarrow \z$ as follows
\begin{align*}
b(e_1,f_1)&=u_1-u_2,~~~~ b(e_2,f_1)=u_1+u_2,~~~~b(e_3,f_1)=u_3,
\\b(e_1,f_2)&=u_1+u_2,~~~~b(e_2,f_2)=-u_1+u_2,~~~~ b(e_3,f_2)=u_4.
\end{align*}
for $h=\sum_{i=1}^3h_ie_i$, and $x=x_1f_1+x_2f_2$, we have
$$b(h,x)=[h_1(x_1+x_2)+h_2(x_1-x_2)]u_1+[h_1(-x_1+x_2)+h_2(x_1+x_2)]u_2+h_3x_1u_3+h_3x_2u_4.$$
and, $$\Vert b(h,x)\Vert^2_{\z}=2(h_1^2+h_2^2)(x_1^2+x_2^2)+h_3^2(x_1^2+x_2^2).$$
Hence
$$\Vert h\Vert_{\h}\Vert x\Vert\leqslant \Vert b(h,x)\Vert_{\z}\leqslant \sqrt{2}\Vert h\Vert_{\h}\Vert x\Vert.$$
In the other hand $\ran(b)=\R^4=\z$.\\
Let $z=\sum_{i=1}^4z_iu_i \in \z$, then
\begin{equation*}
\langle b(h,x),z\rangle_{\z}=\langle h,\langle z/x\rangle\rangle_{\h}.
\end{equation*}
and also
$$\langle b(h,x),z\rangle_{\z}=h_1[(x_1+x_2)z_1+(-x_1+x_2)z_2]+h_2[(x_1-x_2)z_1+(x_1+x_2)z_2]+h_3[x_1z_3+x_2z_4].$$
It follows that, $$\langle z/x\rangle=[(x_1+x_2)z_1+(-x_1+x_2)z_2]e_1+[(x_1-x_2)z_1+(x_1+x_2)z_2]e_2+[x_1z_3+x_2z_4]e_3.$$
Choose $x=f_1+f_2$, we obtain $$\langle z/x\rangle=2z_1e_1+2z_2e_2+(z_3+z_4)e_3.$$
Hence
$$ \Vert\langle z/x\rangle\Vert^2_{\h}=4 z_1^2+4 z_2^2+ z_3^2+ z_4^2+\vert z_3z_4\vert\leqslant 4\Vert z\Vert^2_{\z}, ~~\forall z\in\z.$$
 Then we can say that $\{x=f_1+f_2\}$ is a $b$-Bessel sequence for $\z=\C^4$. Now consider the linear orthogonal projection on $\C u_1\oplus\C u_2$: \begin{align*}
K: \z&\rightarrow \z\\ (z_1,z_2,z_3,z_4)&\mapsto (z_1,z_2)
\end{align*}
Then we easily get
$$ \Vert K^*z\Vert^2_{\z}\leqslant\Vert\langle z/x\rangle\Vert^2_{\h}\leqslant4\Vert z\Vert^2_{\z}.$$
And $\{x=f_1+f_2\}$ becomes a $K$-$b$-frame for $\z$ with bounds $A=1$ and $B=4$.
\end{ex}

\begin{lem}\label{lem}
If $\{x_i\}_{i\in\N}$ is a $\flat$-orthonormal basis for $\z$, and $z\in\z$, such that $z=\underset{i\in\N}{\sum}\flat(h_i,x_i)$, $h_i\in\h,\forall i \in \N$, then 
$$\sum_{i\in\N}\Vert h_i\Vert^2_{\h}=\Vert z\Vert^2_{\z}.$$
\end{lem}

\begin{proof}
$\{x_i\}_{i\in\N}$ is $\flat$-orthonormal, then for every $i,j$ in $\N$ we have $\langle \flat(h_i,x_i)/x_i\rangle= \delta_i^jh_i$.\\
Hence $$\left\langle \underset{i\in\N}{\sum}\flat(h_i,x_i)/x_j\right\rangle= \sum_{i\in\N}\delta_i^jh_i.$$
 And as $\{x_i\}_{i\in\N}$ is a $\flat$-basis for $\z$, then every $z\in\z$ is uniquely represented as $z=\underset{i\in\N}{\sum}\flat(h_i,x_i)$. This implies that $\langle z/x\rangle = \underset{i\in\N}{\sum}\delta_i^jh_i$.\\
  For $i=j$, we have $\langle z/x_j\rangle=h_j$. Hence
$$\langle h_j,\langle z/x_j\rangle\rangle_{\h}=\Vert h_j\Vert^2_{\h},$$
$\Rightarrow$ 
$$\langle \flat(h_j,x_j),z\rangle_{\z}=\Vert h_j\Vert^2_{\h},$$
$\Rightarrow$ 
$$\left\langle \sum_{j\in\N} \flat(h_j,x_j),z\right\rangle_{\z} = \sum_{j\in\N}\Vert h_j\Vert^2_{\h}=\Vert z\Vert^2_{\z}.$$
\end{proof}

\begin{thm}\label{bessel=Tdef}
A sequence $\{x_i\}_{i\in\N} \subset \B$ is $K$-$\flat$-Besselian in $\z$ with bound $B$ if and only if the operator 
\begin{align*}
T: l_2(\h) &\rightarrow \z\\
\{h_i\}_{i\in\N} &\mapsto \sum_{i\in\N} \flat(h_i,x_i)
\end{align*}
is well defined, bounded, has closed range, and $\Vert T\Vert \leq \sqrt{B}.$ Moreover , its adjoint operator $T^*: \z \longrightarrow l_2(\h)$ is determined by $T^*(z)= \{\left\langle z/ x_i\right\rangle\}_{i\in\N}$.
\end{thm}

\begin{proof}
Let $\{x_i\}_{i\in\N}$ be $K$-$\flat$-Besselian in $\z$ with bound $B$. The serie $\sum_{i\in\N} \flat(h_i,x_i)$ converges for every $\{h_i\}_{i\in\N}\in l_2(\h)$. Let $n,m\in \N$, and set $z_n=\sum_{i=1}^n \flat(h_i,x_i)$, we have

\begin{align*}
\Vert z_m -z_n\Vert_{\z} &= \sup_{\Vert z\Vert=1}\left\vert\left\langle \sum_{i=n+1}^m \flat(h_i,x_i),z\right\rangle_{\z}\right\vert\\
&=\sup_{\Vert z\Vert=1}\left\vert \sum_{i=n+1}^m\left\langle h_i,\langle z/x_i\right\rangle\rangle_{\h}\right\vert\\
&\leqslant \sup_{\Vert z\Vert=1} \sum_{i=n+1}^m \Vert h_i\Vert_{\h}\Vert\langle z/x_i\rangle\Vert_{\h}\\
&\leqslant \sup_{\Vert z\Vert=1} \left( \sum_{i=n+1}^m \Vert \langle z/ x_i\rangle\Vert^2_{\h}\right)^{\frac{1}{2}}\left( \sum_{i=n+1}^m \Vert h_i\Vert^2_{\h}\right)^{\frac{1}{2}}\leqslant \sqrt{B}\left(\sum_{i=n+1}^m \Vert h_i\Vert^2_{\h}\right)^{\frac{1}{2}}.
\end{align*}
Note that $\left(\sum_{i=n+1}^m \Vert h_i\Vert^2_{\h}\right)^{\frac{1}{2}}$ tends to $0$ when $n$ tends to $\infty$, hence $\{z_n\}_{n\in\N}$ verifies Cauchy criterion, then it converges. We claim that the operator $T$ is well defined, and $\Vert T\Vert\leq \sqrt{B}$.\\
Conversely, suppose that $T$ is well defined, and that $\Vert T\Vert \leq \sqrt{B}$. Let $\{h_i\}_{i\in\N} \subset l_2(\h)$ and $z\in\z$, we have
\begin{align*}
\left\langle T(\{h_i\}_{i\in\N}), z\right\rangle_{\z} &=\sum_{i\in\N} \langle \flat(h_i,x_i),z\rangle_{\z}\\
&=\sum_{i\in\N}\langle h_i, \langle z/x_i\rangle\rangle_{\h}\\
&=\left\langle\{h_i\}_{i\in\N}, \{\langle z/ x_i\rangle\}_{i\in\N}\right\rangle_2.
\end{align*}
which gives us the expression of the adjoint operator $$T^*(z)= \{\langle z/ x_i\rangle\}_{i\in\N}, ~~~\forall z\in\z.$$
We have 
\begin{align*}
\sum_{i\in\N} \Vert \langle z/ x_i\rangle\Vert^2_{\h} &= \Vert T^*(z)\Vert^2_2 \leq \Vert T^*\Vert^2\Vert z\Vert^2_{\z}\\
&=\Vert T\Vert^2\Vert z\Vert^2_{\z}\leq B\Vert z\Vert^2_{\z}.
\end{align*}
hence, $\{x_i\}_{i\in\N}$ is $K$-$\flat$-Besselian.
\end{proof}

$T$ is called $K$-$\flat$-synthesis operator, and $T^*$ is called $K$-$\flat$-analysis operator.

\begin{defi}
Suppose that $\{x_i\}_{i\in\N} \subset \B$ is a $K$-$\flat$-frame for $\z$ with bounds $A$ and $B$, then the operator $S: \z \longrightarrow \z$ defined by:
\begin{equation}
S(z)=\sum_{i\in\N} \flat\left(\langle z/x_i\rangle, x_i\right), ~~~\forall z \in \z.
\end{equation}
is called $K$-$\flat$-frame operator for $\{x_i\}_{i\in\N}$, and $S=TT^*$.
\end{defi}

\begin{thm}
Let $\{x_i\}_{i\in\N} \subset \B$ be a $K$-$\flat$-frame for $\z$ with bounds $A,B$ and with $K$-$\flat$-frame operator $S$. Then $S$ is bounded, positive, and self-adjoint in $\z$.
\end{thm}

\begin{proof}
for all $z \in \z$, 
\begin{equation}\label{s=fr}
\langle Sz, z\rangle_{\z}=\left\langle \sum_{i\in\N} \flat\left(\langle z/x_i\rangle, x_i\right),z\right\rangle_{\z}= \sum_{i\in\N} \left\langle\flat\left(\langle z/x_i\rangle, x_i\right),z\right\rangle_{\z}=\sum_{i\in\N} \Vert\langle z/x_i\rangle\Vert^2_{\h};
\end{equation}
which means that $S$ is positive, and as $S=TT^*$, $S^*=(TT^*)^*=TT^*$, so $S$ is self-adjoint.\\
Finally, if we gather (\ref{k-b-frame}) and (\ref{s=fr}) we obtain that, $\exists 0<A\leq B<\infty$, such that
$$A\langle K^*z,K^*z\rangle_{\z}=A\langle KK^*z,z\rangle_{\z}\leq \langle Sz,z\rangle_{\z}\leq B\langle z,z\rangle_{\z}, ~~ \forall z \in \ran(K),$$
hence $$AKK^*\leq S\leq BI_{\z}.$$
\end{proof}

\begin{pro}\label{Kclosedran}
Let $\{x_i\}_{i\in\N} \subset \B$ be a $K$-$\flat$-frame for $\z$ with bounds $A$ and $B$. If $\ran(K)$ is closed, then $S$ is invertible on $\ran(K)$, and moreover $$B^{-1}\leq \Vert S^{-1}\Vert\leqslant A^{-1}\Vert K^{\dagger}\Vert^2.$$
\end{pro}

\begin{proof}
Suppose that $\ran(K)$ is closed, then there exists $K^{\dagger} \in \el(\z)$ such that 
$$ KK^{\dagger}z=z,~~~\forall z \in\ran(K),$$ and by (\ref{s=fr}) we have
$$\langle Sz, z\rangle_{\z}=\sum_{i\in\N} \Vert\langle z/x_i\rangle\Vert^2_{\h} \geq A\Vert K^*z\Vert_{\z}^2.$$
And also $\langle Sz, z\rangle_{\z}\leq \Vert Sz\Vert_{\z}\Vert z\Vert_{\z}$, thus $A\Vert K^*z\Vert_{\z}^2\leq \Vert Sz\Vert_{\z}\Vert z\Vert_{\z}$.\\ In the other hand one has $$\Vert z\Vert_{\z} = \Vert KK^{\dagger}z\Vert_{\z}
=\Vert (K^{\dagger})^*K^*z\Vert_{\z},~~ (since~(KK^{\dagger})^*=Id_{\ran(K)}).$$

Then we obtain for all $z\in \ran(K)$,

$$\Vert z\Vert_{\z}^2 \leq \Vert (K^{\dagger})^*\Vert^2 \Vert K^*z\Vert_{\z}^2=\Vert K^{\dagger}\Vert^2 \Vert K^*z\Vert_{\z}^2,$$
$\Rightarrow$
$$ A \Vert K^{\dagger}\Vert^{-2} \Vert z\Vert_{\z}^2 \leq A \Vert K^*z\Vert_{\z}^2 \leq \Vert Sz\Vert_{\z}\Vert z\Vert_{\z},$$
$\Rightarrow$
$$ A \Vert K^{\dagger}\Vert^{-2}\Vert z\Vert_{\z} \leq \Vert Sz\Vert_{\z}.$$
hence by iii) of theorem \ref{prop de T}, $S^*$ is surjective, and $S$ is injective, but as $S$ is self-adjoint, then $S$ is invertible on $\ran(K)$. We also have $$\Vert SS^{-1}z\Vert_{\z}=\Vert z\Vert_{\z}\leqslant  B\Vert S^{-1}z\Vert_{\z}\leqslant  B\Vert S^{-1}\Vert\Vert z\Vert_{\z} \Rightarrow \Vert S^{-1}\Vert \geq  B^{-1}.$$
Finally we have shown that 
$$A \Vert K^{\dagger}\Vert_{\el(\z)}^{-2}\Vert z\Vert_{\z} \leq \Vert Sz\Vert_{\z},$$
 hence $$A \Vert K^{\dagger}\Vert^{-2}\Vert S^{-1}z\Vert_{\z} \leq \Vert SS^{-1}z\Vert_{\z}.$$
  We conclude that $$\Vert S^{-1}\Vert\leqslant A^{-1}\Vert K^{\dagger}\Vert^2,~~\forall z\in\ran(K).$$
\end{proof}
\begin{rmq}
When $S$ is invertible on $\ran(K)$ we have 
$$SS^{-1}(z)=S^{-1}S(z)=z,~~ (\forall z\in\ran(K)).$$
$\Leftrightarrow$
$$z=\sum_{i\in\N} \flat\left(\left\langle S^{-1}(z)/x_i\right\rangle,x_i\right),~~ (\forall z\in\ran(K)).$$
$\Leftrightarrow$
$$z=\sum_{i\in\N} S^{-1}\left(\flat(\langle z/x_i\rangle,x_i\right),~~ (\forall z\in\ran(K)).$$
In the case where $\{x_i\}_{i\in\N}$ is a $\flat$-frame, then 
$$z=\sum_{i\in\N} \flat\left(\left\langle S^{-1}(z)/x_i\right\rangle,x_i\right)=\sum_{i\in\N} S^{-1}\left(\flat(\langle z/x_i\rangle,x_i\right),~~ (\forall z\in\z).$$
\end{rmq}

\begin{thm}\label{Tdefsurj=kbframe}
Let $K \in \el(\z)$ has closed range. The sequence $\{x_i\}_{i\in\N} \subset \B$ forms a $K$-$\flat$-frame for $\ran(K)$ if and only if the operator,
\begin{align*}
T: l_2(\h) &\rightarrow \z\\
\{h_i\}_{i\in\N} &\mapsto \sum_{i\in\N} \flat(h_i,x_i)
\end{align*}
is defined, bounded and surjective on $\ran(K)$.
\end{thm}

\begin{proof}
Let $\{x_i\}_{i\in\N} \subset \B$ be a $K$-$\flat$-frame for $\z$,  then it is $K$-$\flat$-Besselian in $\z$. Hence, by theorem \ref{bessel=Tdef} $T$ is well defined, and let $S=TT^*$ be the corresponding $K$-$\flat$- frame operator, we proved that $S$ is surjective on $\ran(K)$, then $T$ is also surjective on $\ran(K)$.\\
Conversely, suppose that $T$ is defined and surjective, then by theorem \ref{bessel=Tdef} $\{x_i\}_{i\in\N}$ is $K$-$\flat$-Besselian, and also $T$ is surjective, then for each $z \in \ran(K)$, $\exists \{y_i\}_{i\in\N}\subset l_2(\h)$ such that $T(\{y_i\})=z$, and let $T^{\dagger}z=\{y_i\}$
\begin{align*}
\Vert K^*z\Vert^4_{\z} = \vert\langle K^*z,K^*z\rangle_{\z}\vert^2 =\vert\langle z,KK^*z\rangle_{\z}\vert^2 &= \left\vert\langle T(\{y_i\}),KK^*z\rangle_{\z}\right\vert^2\\
&= \left\vert\left\langle \sum_{i=1}^{\infty} b(y_i,x_i),KK^*z\right\rangle_{\z}\right\vert^2\\
&\leqslant\Vert K\Vert^4\left\vert\sum_{i=1}^{\infty}\left\langle y_i ,\langle z/x_i\rangle \right\rangle_{\h}\right\vert^2_{\h}\\
&\leqslant\Vert K\Vert^4\left\vert \sum_{i=1}^{\infty} \Vert y_i\Vert_{\h}\Vert \langle z/x_i\rangle\Vert_{\h}\right\vert^2\\
&\leqslant\Vert K\Vert^4\Vert T^{\dagger}z\Vert^2_{\ell^2(\h)}\sum_{i=1}^{\infty}\Vert \langle z/x_i\rangle\Vert^2_{\h}\\
&\leqslant \Vert K\Vert^4\Vert T^{\dagger}\Vert^2\Vert z\Vert_{\z}^2\sum_{i=1}^{\infty}\Vert\langle z/x_i\rangle\Vert^2_{\h}\\
&\leqslant\Vert K\Vert^4\Vert T^{\dagger}\Vert^2\Vert K^{\dagger}\Vert^2\Vert K^*z\Vert_{\z}^2\sum_{i=1}^{\infty}\Vert\langle z/x_i\rangle\Vert^2_{\h}~~ \forall z \in \ran(K).
\end{align*}
Hence \begin{equation}\label{equa}
\Vert K\Vert^{-4}\Vert T^{\dagger}\Vert^{-2}\Vert K^{\dagger}\Vert^{-2}\Vert K^*z\Vert^2_{\z}\leqslant \sum_{i=1}^{\infty}\Vert\langle z/x_i\rangle\Vert^2_{\h},~~\forall z\in\ran(K).
\end{equation}
Hence the proof holds.
\end{proof}

\begin{thm}\label{b-f=b-orthbase}
Let $K\in\el(\z)$, such that $\ran(K)$ is closed. A sequence $\{x_i\}_{i\in\N}\subset\B$ is a $K$-$\flat$-frame for $\ran(K)$ with bounds $A$ and $B$, if and only if the two following conditions
are satisfied:
\begin{enumerate}
\item[1)]$\{x_i\}_{i\in\N}$ is $\flat$-complete in $\ran(K)$.
\item[2)] The $K$-$\flat$-synthesis operator is well defined and 
\begin{equation}
\frac{A}{\Vert K^{\dagger}\Vert^2}\sum_{i\in\N}^{\infty}\Vert h_i\Vert^2_{\h}\leqslant \Vert T(\{h_i\}_{i\in\N})\Vert^2_{\z}\leqslant B \sum_{i\in\N}^{\infty}\Vert h_i\Vert^2_{\h}, ~~~\forall \{h_i\}_{i\in\N}\in\h.
\end{equation}
\end{enumerate}
\end{thm}

\begin{proof}
\begin{itemize}
\item[$(\Rightarrow)$]
Suppose that $\{x_i\}_{i\in\N}$ is a $K$-$\flat$-frame for $\ran(K)$ with bounds $A$ and $B$. And suppose that $\langle z/x_i\rangle=0, \forall i\in\N$, we have
$$A \Vert K^*z\Vert^2_{\z} \leqslant \sum_{i=1}^{\infty}\Vert \left\langle z/ x_i\right\rangle\Vert^2_{\h} \leqslant B\Vert z\Vert^2_{\z}, ~~~\forall z \in \ran(K).$$
which means that $K^*z=0$. $\ran(K)$ is closed, then there exists $K^{\dagger}\in\el(\z)$ such that $KK^{\dagger}z=z,~\forall z\in\ran(K)$. We have
$$
\Vert z\Vert^2_{\z}= \vert\langle z,z\rangle_{\z}\vert=\vert\langle KK^{\dagger}z,z\rangle_{\z}\vert=\vert\langle K^{\dagger}z,K^*z\rangle_{\z}\vert=\vert\langle K^{\dagger}z,0\rangle_{\z}\vert=0.$$ 
Hence $z=0$, which implies that $\{x_i\}_{i\in\N}$ is $\flat$-complete on $\ran(K)$. In addition, we claim that for every $z\in\ran(K)$, there exists a unique $\{h_i\}_{i\in\N}\subset \h$, such that $z=\sum_{i\in\N}^{\infty}b(h_i,x_i)=T(\{h_i\}_{i\in\N})$, as long as $T$ is surjective, and that by Lemma \ref{lem} we have
$$\frac{A}{\Vert K^{\dagger}\Vert^2}\sum_{i\in\N}^{\infty}\Vert h_i\Vert^2_{\h}\leqslant A \Vert K^*z\Vert^2_{\z} \leqslant \sum_{j\in\N}^{\infty}\Vert \langle\sum_{i\in\N}^{\infty} b(h_i,x_i)/x_j\rangle\Vert^2_{\h}\leqslant B \sum_{i\in\N}^{\infty}\Vert h_i\Vert^2_{\h}, ~~~\forall \{h_i\}_{i\in\N}\in\h$$
hence
$$\frac{A}{\Vert K^{\dagger}\Vert^2}\sum_{i\in\N}^{\infty}\Vert h_i\Vert^2_{\h}\leqslant \sum_{j\in\N}^{\infty}\Vert h_j\Vert^2_{\h}\leqslant B \sum_{i\in\N}^{\infty}\Vert h_i\Vert^2_{\h}, ~~~\forall \{h_i\}_{i\in\N}\in\h$$
$\Rightarrow$ $$\frac{A}{\Vert K^{\dagger}\Vert^2}\sum_{i\in\N}^{\infty}\Vert h_i\Vert^2_{\h}\leqslant \Vert z\Vert^2_{\h}\leqslant B \sum_{i\in\N}^{\infty}\Vert h_i\Vert^2_{\z}, ~~~\forall \{h_i\}_{i\in\N}\in\h$$
$\Rightarrow$ $$\frac{A}{\Vert K^{\dagger}\Vert^2}\sum_{i\in\N}^{\infty}\Vert h_i\Vert^2_{\h}\leqslant \sum_{j\in\N}^{\infty}\Vert T(\{h_i\}_{i\in\N})\Vert^2_{\h}\leqslant B \sum_{i\in\N}^{\infty}\Vert h_i\Vert^2_{\h}, ~~~\forall \{h_i\}_{i\in\N}\in\h.$$
 
\item[$(\Leftarrow)$] Conversely, suppose that $1)$ and $2)$ hold, by theorem \ref{Tdefsurj=kbframe}, we only need to show that $T$ is surjective on $\ran(K)$. Suppose that $T^*(z)=0$, then $\Vert T^*(z) \Vert_{\h}=\underset{i\in\N}{\sum}\Vert \langle z/x_i\rangle\Vert^2_{\h}=0,~~\forall i\in\N$. This implies that $\Vert z\Vert_{\ran(K)}=0\Rightarrow z=0$, we conclude that $T^*$ is injective, which means that $T$ is surjective on $\ran(K)$.
\end{itemize}
\end{proof}

\subsection{The $\flat$-adjoint operator}
In the study of Hilbert spaces, the operator theory took an important place, where bounded operators played an important role. More specifically the adjoint operator or also called the Hermitian conjugate. Now we want to give a generalization of this notion according to the $\flat$-dual product defined in proposition\ref{dualprod} in order to give some results about $b$-frames preserving. In other words we want to figure out under which conditions we have for a bounded linear operator $U$ on a banach space $\B$, the existence of a unique bounded operator $V$ on a Hilbert space $\z$ verifying 
$$\langle z/Ux\rangle=\langle Vz/x\rangle, ~~~\forall x\in\B,~\forall z\in\z.$$
From now on we will suppose that $\h$ is separable and $\{e_i\}_{i\in\N}$ is its orthonormal basis, and that the bilinear mapping $\flat$ has a dense range, and that $\{x_i\}_{i\in\N}\subset \B$ is a $\flat$-orthonormal basis in $\z$ such that $$\sum_{i\in\N}\Vert x_i\Vert<\infty.$$\\
Denote by $\X$ the subspace of $\B$ such that $$\X=\overline{span\{x_i/i\in\N\}}.$$

\begin{pro}
Let $U:\X\rightarrow\X$ be a bounded linear operator, let $\{h_i\}_{i\in\N} \subset\h$, and let $V: \z\rightarrow \z$ be an operator such that for each $z=\underset{i\in\N}{\sum} b(h_i,x_i) \in\z$, $$V^*(z)=\sum_{i\in\N}= \flat(h_i,Ux_i),$$
Then $V^*$ is well defined, bounded and linear.
\end{pro}

\begin{proof}
$V^*$ is well defined on $\z$ since $\{x_i\}_{i\in\N}\subset \B$ is a $\flat$-orthonormal basis in $\z$ and it is linear, indeed: for each $z=\underset{i\in\N}{\sum}\flat(h_i,x_i),$ and $z'=\underset{i\in\N}{\sum}\flat(h'_i,x_i) \in\z$ and let $\gamma\in\C$, we have 

\begin{align*}
V^*(\gamma z+z')&=V^*\left(\gamma\sum_{i\in\N} \flat(h_i,x_i)+\sum_{i\in\N} \flat(h'_i,x_i)\right)\\
&=V^*\left(\sum_{i\in\N}\flat(\gamma h_i+h'i, x_i)\right)\\
&=\sum_{i\in\N}\flat(\gamma h_i+h'i,U x_i)\\
&=\gamma V^*(z)+V^*(z')
\end{align*}
$V^*$ is also bounded, in fact we have for all $z=\underset{i\in\N}{\sum}\flat(h_i,x_i)\in \z$,

\begin{align*}
\Vert V^*z\Vert_{\z}=\left\Vert V^*\left(\sum_{i\in\N}\flat(h_i,x_i)\right)\right\Vert_{\z}&=\left\Vert \sum_{i\in\N}V^*(\flat(h_i,x_i))\right\Vert_{\z}\\
&= \left\Vert \sum_{i\in\N}\flat(h_i,Ux_i)\right\Vert_{\z}\\
&\leqslant \sum_{i\in\N}\Vert \flat(h_i,Ux_i)\Vert_{\z}\\
& \leqslant \beta\Vert U\Vert\sum_{i\in\N}\Vert h_i\Vert_{\h} \Vert x_i\Vert\\
&\leqslant \beta\Vert U\Vert \sup_{i\in\N}\Vert h_i\Vert_{\h} \sum_{i\in\N} \Vert x_i\Vert
\end{align*}
we have already shown that for all $z=\underset{i\in\N}{\sum}\flat(h_i,x_i)\in \z$, $$\Vert z\Vert^2_{\z}=\sum_{i\in\N}\Vert h_i\Vert^2_{\h},$$
hence $\Vert z\Vert^2_{\z} \geqslant \sup_{i\in\N}\Vert h_i\Vert^2_{\h}=(\sup_{i\in\N}\Vert h_i\Vert_{\h})^2$, which means that $\Vert z\Vert_{\z} \geqslant \sup_{i\in\N}\Vert h_i\Vert_{\h}$. In consequence we have 
$$ \Vert V^*z\Vert_{\z} \leqslant \beta \Vert U\Vert \sum_{i\in\N}\Vert x_i\Vert \Vert z\Vert_{\z}<\infty.$$
\end{proof}
\begin{thm}\label{exstadj}
Let $\{x_i\}_{i\in\N}\subset \B$ be a $\flat-$orthonormal basis in $\z$. Let $U:\X\rightarrow\X$ be a bounded linear operator, there exists a unique operator $V\in\el(\z)$ such that, $\forall z\in\z$,
\begin{equation}\label{adj}
 \langle Vz/x_i\rangle=\langle z/Ux_i\rangle,~~~\forall i\in\N.
 \end{equation}
\end{thm}

\begin{proof}
Let's provide $\h$ with its orthonormal basis $\{e_j\}_{j\in\N}$, and let $h=\sum_{j\in\N} \alpha_je_j \in\h$, note that if $z_k=\sum_{j\in\N}\flat(h_j,x_j)$ with $h_j=0$ if $j\neq k$ and $h_k=h$, then $z_k$ can be also written as $z_k=\flat(h,x_k)=\sum_{j\in\N} \alpha_j\flat(e_j,x_k)$
Let $V$ the adjoint operator of $V^*$ on $\z$. Let $z\in \z$, we have
\begin{equation}\label{adj1}
\langle V^*(\flat(h,x_i)),z\rangle_{\z}=\langle \flat(h,Ux_i),z\rangle_{\z}= \langle h,\langle z/Ux_i\rangle\rangle_{\h}.
\end{equation}
In the other hand
\begin{equation}\label{adj2}
\langle V^*(\flat(h,x_i)),z\rangle_{\z}=\langle \flat(h,x_i),Vz\rangle_{\z}= \langle h,\langle Vz/x_i\rangle\rangle_{\h}.
\end{equation}
from (\ref{adj1}) and (\ref{adj2}) we obtain for all $h\in\h$,
$$ \langle h,\langle z/Ux_i\rangle\rangle_{\h} = \langle h,\langle Vz/x_i\rangle\rangle_{\h},$$
$\Rightarrow$
$$\forall i\in\N,~~\langle z/Ux_i\rangle=\langle Vz/x_i\rangle.$$
Suppose now that there exists $W$ also verifying (\ref{adj}), then we have 
$$
\langle z/Ux_i\rangle=\langle Vz/x_i\rangle=\langle Wz/x_i\rangle,~~\forall i\in\N,$$$\Rightarrow$
$$\langle (V-W)z/x_i\rangle=0,~~\forall i\in\N,$$
$\Rightarrow$
$$V=W,~~\text{on }\X.$$
hence $V$ is unique.
\end{proof}

\begin{cor}\label{bdu}
Let $\{x_i\}_{i\in\N}\subset \B$ be a $\flat-$orthonormal basis in $\z$. Let $U:\X\rightarrow\X$ be a bounded linear operator, there exists a unique operator $V\in\el(\z)$ such that,
\begin{equation}\label{adj}
 \langle Vz/x\rangle=\langle z/Ux\rangle,~~~\forall x\in\X,\forall z\in\z.
 \end{equation}
\end{cor}

We call the unique operator verifying (\ref{adj}) on $\z$,  the $\flat$-adjoint of $U$ and we denote it by $U^{\flat}$.

\begin{ex}
If we go back to the example (\ref{exemple b-orth}) and we consider the bounded linear operator
$ U:\B \rightarrow\B$ such that $U(f_1)=f_1+f_2$ and $U(f_2)=0$.
We have
$$\langle b(h,Ux),z\rangle=\langle b(h,x_1f_1+x_1f_2),z\rangle=h_1x_1z_1+h_2x_1z_3.$$
So, $\langle z/Ux\rangle=x_1z_1e_1+x_1z_3e_2.$\\
Let $V:\z\rightarrow\z$ such that $\langle Vz/x\rangle=x_1z_1e_1+x_1z_3e_2$, we can write $V$ as 
$$ Vz= \left(\begin{matrix}
\alpha^1_1 &\alpha^1_2&\alpha^1_3\\
\alpha^2_1 &\alpha^2_2&\alpha^2_3\\
\alpha^3_1 &\alpha^3_2&\alpha^3_3
\end{matrix}\right)\left(\begin{matrix}
z_1\\ z_2\\z_3
\end{matrix}\right)=\sum_{i=1}^3\alpha^1_iz_iu_1+\sum_{i=1}^3\alpha^2_iz_iu_2+\sum_{i=1}^3\alpha^3_iz_iu_3.$$
Recall that, $\langle z/x\rangle=(x_1z_1+x_2z_2)e_1+x_1z_3e_2$, so $$\langle Vz/x\rangle=\left(x_1\sum_{i=1}^3\alpha^1_iz_i+x_2\sum_{i=1}^3\alpha^2_iz_i\right)e_1+x_1\sum_{i=1}^3\alpha^3_iz_ie_2.$$
Then
$$\left(x_1\sum_{i=1}^3\alpha^1_iz_i+x_2\sum_{i=1}^3\alpha^2_iz_i\right)e_1+x_1\sum_{i=1}^3\alpha^3_iz_ie_2=x_1z_1e_1+x_1z_3e_2.$$
By a simple calculus we obtain the $b$-adjoint of $U$,
$$U^b=V=\left( \begin{matrix}
1&0&0\\0&0&0\\0&0&1
\end{matrix}\right).$$
Hence, $U^bz=z_1u_1+z_3u_3,$ $\forall z\in\z.$
\end{ex}

\begin{ex}
Let $\h=\overline{\ell^2(\C)\cap\ell^1(\C)}^{\ell^2(\C)}$, $\z=\ell^2(\C)$, and $\B=\ell^{\infty}(\C)$. And let
\begin{align*}
U:\B&\rightarrow\B\\
\{x_n\}_{n\in\N} &\mapsto \{x_{n+1}\}_{n\in\N}
\end{align*}
we have $\Vert U\Vert=1$, and consider the bilinear mapping
\begin{align*}
\phi: \h\times\B &\rightarrow\z\\
\left(\{h_n\}_{n\in\N},\{b_n\}_{n\in\N}\right)&\mapsto \{h_nb_n\}_{n\in\N}
\end{align*}
It's clear that: $\Vert \phi\left(\{h_n\}_{n\in\N},\{b_n\}_{n\in\N}\right)\Vert_2\leq \Vert \{h_n\}_{n\in\N}\Vert_{\h}\Vert \{b_n\}_{n\in\N}\Vert_{\infty}.$\\
And we have for every $\{z_n\}_{n\in\N}\subset\z$
$$
\langle \phi\left(\{h_n\}_{n\in\N},\{b_n\}_{n\in\N}\right),\{z_n\}_{n\in\N}\rangle_2= \{h_nb_n\bar{z_n}\}_{n\in\N}=\{h_n\overline{\bar{b_n}z_n}\}_{n\in\N}
=\langle \{h_n\}_{n\in\N},\{\bar{b_n}z_n\}_{n\in\N}\rangle_{\h}.$$
hence for every $\{z_n\}_{n\in\N}\subset\z$ and every $\{b_n\}_{n\in\N}\subset\B$, we have, $ \langle \{z_n\}_{n\in\N},\{b_n\}\rangle_{\h}=\{\bar{b_n}z_n\}_{n\in\N}.$
Whence 
\begin{align*}
\langle \phi\left(\{h_n\}_{n\in\N},U(\{b_n\}_{n\in\N})\right),\{z_n\}_{n\in\N}\rangle_2
&=\langle \{h_n\}_{n\in\N},\{\overline{b_{n+1}}z_n\}_{n\in\N}\rangle_{\h}\\
&=\langle \{h_n\}_{n\in\N},\{\sigma_n\}_{n\in\N^*}\rangle_{\h}\\
&=\langle \{h_n\}_{n\in\N},\langle V(\{z_n\}_{n\in\N})/\{b_n\}_{n\in\N}\rangle\rangle_{\h}
\end{align*}
Where $\sigma_k=\bar{b}_kz_{k-1}$, for $k>0$, and $\sigma_0=0$, and
\begin{align*}
V:\z&\rightarrow\z\\
\{z_n\}_{n\in\N} &\mapsto \{0,z_0,z_1,...,z_n,..)
\end{align*}
$V=U^{\flat}$ is the $\flat$-adjoint operator of $U$ on $\ell^2(\C)$.
(Note that the $b$-adjoint of $U$ in $\ell^2(\C)$ coincides with the adjoint operator of $U$ in $\ell^{\infty}(\C)$).
\end{ex}

\begin{rmq}
The opposite of Corollary \ref{bdu} is not true; in fact, for a bounded linear operator $V:\z\rightarrow\z$ we cannot always find a bounded linear operator $U:\B\rightarrow\B$ such that $$\langle Vz/x\rangle=\langle z/Ux\rangle,~~~\forall z\in\z,\forall x\in\B.$$
\end{rmq}
We can see this by giving the simple following example:
\begin{ex}
Let $\h=\R^3$, $\B=\R^2$, and $\z=\R^4$, and let $(e_1,e_2,e_3), (f_1,f_2)$, and $(u_1,u_2,u_3,u_4)$ be their canonical bases respectively. Consider the bilinear mapping $b:\h\times\B\rightarrow\z$ such that
\begin{align*}
b(e_1,f_1)&=u_1,~~~~ b(e_2,f_1)=u_3,~~~~b(e_3,f_1)=u_1-u_2,
\\b(e_1,f_2)&=u_2,~~~~b(e_2,f_2)=u_4,~~~~ b(e_3,f_2)=u_1+u_2.
\end{align*}
Let $h=\sum_{i=1}^3h_ie_i\in\h$ and $x=x_1f_1+x_2f_2\in\B$, we have $$ b(h,x)=[(h_1+h_3)x_1+h_3x_2]u_1+[(h_1+h_3)x_2-h_3x_1]u_2+h_2x_1u_3+h_2x_2u_4,$$
and for each $z=\sum_{i=1}^4z_iu_i\in\z$, we have $$ \langle b(h,x),z\rangle= h_1(x_1z_1+x_2z_2)+h_2(x_1z_3+x_2z_4)+h_3[(x_1+x_2)z_1+(x_2-x_1)z_2],$$
So we claim that 
$$\langle z/x\rangle=(x_1z_1+x_2z_2)e_1+(x_1z_3+x_2z_4)e_2+[(x_1+x_2)z_1+(x_2-x_1)z_2]e_3, ~~~\forall z\in\z,\forall x\in\B.$$
Now consider the bounded linear operator $V:\z\rightarrow\z$, defined as $V(u_1)=u_2$ and $V(u_2)=V(u_3)=V(u_4)=0$.
We have $\langle Vz/x\rangle=x_2z_1e_1+(x_2-x_1)z_1e_3$. Suppose that there exists a bounded linear operator $U:\B\rightarrow\B$ such that $\langle Vz/x\rangle=\langle z/Ux\rangle$. $U$ will be such as $U(x_1)=\alpha x_1+\beta x_2$ and $U(x_2)=\gamma x_1+\delta x_2$. Then we have
\begin{multline*}
\langle z/Ux\rangle=[(\alpha x_1+\beta x_2)z_1+(\gamma x_1+\delta x_2)z_2]e_1+[(\alpha x_1+\beta x_2)z_3+(\gamma x_1+\delta x_2)z_4]e_2\\+\left([(\alpha+\gamma)x_1+(\beta+\delta) x_2]z_1+[(\alpha-\gamma)x_1+(\beta-\delta)x_2]z_2\right)e_3.
\end{multline*}
We obtain $\alpha=\delta=\gamma=0$ and $\beta=1$, which is contradictory; because if it's not then we will have $\beta x_2z_3e_2\neq0,\forall x_2\in\B,\forall z_3\in\z$, which is not true. Hence we can't find $U$.

\end{ex}

\subsection{Stability and Preserving}

Let $\{x_i\}_{i\in\N},\{y_i\}_{i\in\N}$ be two sequences in $\B$, and suppose that the bilinear mapping $\flat$ has a dense range. We have the following results:
\begin{pro}(Stability by summing).\label{summing}\\
If $\{x_i\}_{i\in\N}$ and $\{y_i\}_{i\in\N}$ are  $\flat-$frames for $\z$, and if $T_x(\{h_i\}_{i\in\N})=\underset{i\in\N}{\sum}\flat(h_i,x_i),$ and $T_y(\{h_i\}_{i\in\N})=\underset{i\in\N}{\sum}\flat(h_i,y_i),$ are respectively their corresponding $\flat$-synthesis operators for $\{h_i\}_{i\in\N}\subset \ell^2(\h)$.\\ Then $\{x_i+y_i\}_{i\in\N}$ is a $\flat-$frame for $\z$ and its $\flat$-synthesis operator is given by $$T_{x+y}(\{h_i\}_{i\in\N})=T_x(\{h_i\}_{i\in\N})+T_y(\{h_i\}_{i\in\N}),~\forall \{h_i\}_{i\in\N}\subset \ell^2(\h).$$ And moreover $$\Vert T_{x+y}\Vert\leq\sqrt{B_x}+\sqrt{B_y}.$$
\end{pro}

\begin{proof}
Suppose that $\{x_i\}_{i\in\N}$ and $\{y_i\}_{i\in\N}$ are  $\flat-$frames for $\z$ with upper bounds $B_x$ and $B_y$ respectively.
Let $\{h_i\}_{i\in\N}\subset \ell^2(\h)$, we have 
\begin{align*}
 T_{x+y}(\{h_i\}_{i\in\N})=T_x(\{h_i\}_{i\in\N})+T_y(\{h_i\}_{i\in\N})&= \sum_{i=1}^{\infty}\flat(h_i,x_i)+\sum_{i=1}^{\infty}\flat(h_i,y_i)\\
 &=\sum_{i=1}^{\infty}\flat(h_i,x_i+y_i).
 \end{align*}
 $T_{x+y}$ is well defined, because we have $$\left\Vert \sum_{i=1}^{\infty}\flat(h_i,x_i+y_i)\right\Vert \leqslant \left\Vert \sum_{i=1}^{\infty}\flat(h_i,x_i)\right\Vert+\left\Vert\sum_{i=1}^{\infty}\flat(h_i,y_i)\right\Vert.$$
And it is clear that $\Vert T_{x+y}\Vert\leq \Vert T_x\Vert+\Vert T_y\Vert\leq \sqrt{B_x}+\sqrt{B_y}<\infty$.\\
Now let $z\in\z$,
$$\langle T_{x+y}(\{h_i\}_{i\in\N}), z\rangle_{\z}= \sum_{i=1}^{\infty}\langle \flat(h_i,x_i+y_i),z\rangle_{\z}
= \sum_{i=1}^{\infty}\langle h_i,\langle z/x_i+y_i\rangle\rangle_{\h}
=\left\langle\{h_i\}_{i\in\N},\{\langle z/x_i+y_i\rangle\}_{i\in\N}\right\rangle_{\h}.$$
hence the $\flat$-analysis operator $T^*$ is given by $T^*(z)=\{\langle z/x_i+y_i\rangle\}_{i\in\N},~\forall z\in\z.$
Now let $z,z'\in\z$, such that $T^*(z)=T^*(z')$, it means that
$$\forall i\in\N, ~~ \langle z/x_i+y_i\rangle =\langle z'/x_i+y_i\rangle,$$
$\Rightarrow$ 
$$\forall h\in\h, \forall i \in\N,~~ \langle b(h,x_i+y_i),z\rangle_{\z}=\langle b(h,x_i+y_i),z'\rangle_{\z},$$
$\Rightarrow$ 
$$\forall h\in\h, \forall i \in\N,~~ \langle b(h,x_i+y_i),z-z'\rangle_{\z}=0.$$

As $\ran(\flat)$ is dense, then $\z=\overline{\ran}(\flat)\oplus\ker(\flat)\Rightarrow \ker(\flat)=\{0\}$, whence necessarily $z=z'$. Which implies that $T$ is surjective, then by theorem 12 of \cite{b-frame} $\{x_i+y_i\}_{i\in\N}$ is a $\flat-$frame for $\z$.
\end{proof}

\begin{cor}
 If $\{x_i\}_{i\in\N}$ and $\{y_i\}_{i\in\N}$ are  $K$-$\flat$-frames for $\z$, and if $T_x(\{h_i\}_{i\in\N})=\underset{i\in\N}{\sum}\flat(h_i,x_i),$ and $T_y(\{h_i\}_{i\in\N})=\underset{i\in\N}{\sum}\flat(h_i,y_i),$ are respectively their corresponding $K$-$\flat$-synthesis operators for $\{h_i\}_{i\in\N}\subset \ell^2(\h)$. Then $\{x_i+y_i\}_{i\in\N}$ is a $K$-$\flat$-frame for $\ran(K)$ and its $K$-$\flat$-synthesis operator is given by $$T_{x+y}(\{h_i\}_{i\in\N})=T_x(\{h_i\}_{i\in\N})+T_y(\{h_i\}_{i\in\N}),~\forall \{h_i\}_{i\in\N}\subset \ell^2(\h)$$ And moreover $\Vert T_{x+y}\Vert\leq\sqrt{B_x}+\sqrt{B_y}$.
\end{cor}

\begin{proof}
The proof is simillar to the one of proposition \ref{summing}, except that in this case the $K$-$\flat$-synthesis operator cannot be surjective unless if $K$ has a closed range and the result won't be valid outside $\ran(K)$.
\end{proof}

\begin{pro}
Let $\{x_i\}_{i\in\N}\subset \B$ be a $K$-$\flat$-frame for $\z$, and let $Q\in \el(\z)$ such that $\ran(Q)\subseteq\ran(K)$. Then $\{x_i\}_{i\in\N}$ is also a $Q$-$\flat$-frame for $\z$.
\end{pro}

\begin{proof}
Suppose that $\{x_i\}_{i\in\N}$ is a $K$-$\flat$-frame for $\z$, then there exist $0<A\leq B<\infty$, such that
$$A \Vert K^*z\Vert^2_{\z} \leqslant \sum_{i=1}^{\infty}\Vert \left\langle z/ x_i\right\rangle\Vert^2_{\h} \leqslant B\Vert z\Vert^2_{\z}, ~~~\forall z \in \ran(K).$$
let $Q \in \el(\h)$, such that $\ran(Q)\subseteq\ran(K)$, then by theorem \ref{douglas}, there exists $\lambda>0$ such that $TT^*\leq \lambda^2 KK^*$, then $\forall z\in\ran(Q)$ we have for $\lambda>0$, $\langle QQ^*z,z\rangle_{\z}\leq \langle \lambda^2KK^*z,z\rangle_{\z}$, which implies that
\begin{align*}
\langle QQ^*z,z\rangle_{\z}\leq \lambda^2\langle KK^*z,z\rangle_{\z}
&\Rightarrow \langle Q^*z,Q^*z\rangle_{\z}\leq \lambda^2\langle K^*z,k^*z\rangle_{\z}\\
&\Rightarrow \Vert Q^*z\Vert^2 \leq \lambda^2\Vert K^*z\Vert^2 \\
&\Rightarrow \frac{A}{\lambda^2}\Vert Q^*z\Vert^2 \leq \Vert K^*z\Vert^2\leq \sum_{i=1}^{\infty}\Vert \left\langle z/ x_i\right\rangle\Vert^2_{\h}\leq B\Vert z\Vert^2_{\z}.
\end{align*}
hence $\{x_i\}_{i\in\N}$ is a $Q$-$\flat
$-frame with bounds $\frac{A}{\lambda^2}$, and $B$.
\end{proof}

\begin{thm}\label{omegax}
Let $\{y_i\}_{i\in\N}\subset\B$ be a $\flat$-orthonormal basis in $\z$, and let $K\in\el(\z)$ have a closed range, then $\{x_i\}_{i\in\N}\subset\B$ forms a $K$-$\flat$-frame for $\ran(K)$ if and only if there exists a bounded surjective operator $\Omega : \z\rightarrow\z$ such that: $$\Omega(\flat(h,y_i))=\flat(h,x_i), ~~~\forall h\in\h,~\forall i\in\N.$$
\end{thm}

\begin{proof}
\begin{itemize} 
\item[$(\Rightarrow)$]Let $\{x_i\}_{i\in\N}\subset\B$, and let $\Omega:\z\rightarrow\z$ be a bounded surjective operator such that $\Omega(\flat(h,y_i))=\flat(h,x_i),~\forall h\in\h,~\forall i\in\N.$ Let $z\in\z$, it is clear that $\langle z/x_i\rangle=\langle \Omega^*z/y_i \rangle .$
Which implies that $$ \sum_{i=1}^{\infty}\Vert \langle z/x_i\rangle \Vert^2_{\h}=\sum_{i=1}^{\infty}\Vert \langle \Omega^*z/y_i \rangle \Vert^2_{\h}=\sum_{i=1}^{\infty}\left\langle\langle \Omega^*z/y_i\rangle,\langle \Omega^*z/y_i\rangle\right\rangle_{\h}=\sum_{i=1}^{\infty}\left\langle\flat(\langle\Omega^*z/y_i\rangle,y_i),\Omega^*z\right\rangle_{\z}=\Vert \Omega^*z\Vert^2_{\z}.$$
First we have $\Vert \Omega^*z\Vert_{\z}\leq \Vert \Omega\Vert\Vert z\Vert_{\z}$. In the other hand, $\Omega$ is surjective, then by iii) of theorem \ref{prop de T}, there exists $A>0$ such that $\Vert \Omega^*z\Vert\geq A\Vert z\Vert_{\z}\geqslant \dfrac{A}{\Vert K^*\Vert}\Vert K^*z\Vert$.
\item[$(\Leftarrow)$] Conversely, suppose that $\{x_i\}_{i\in\N}\subset\B$ is a $K$-$\flat$-frame for $\ran(K)$ then by theorem \ref{Tdefsurj=kbframe} the synthesis operator  $$T\left( \{h_i\}_{i\in\N}\right) = \sum_{i\in\N} \flat(h_i,x_i), ~~\forall h_i\in\h.$$
is well defined and surjective on $\ran(K)$.
 Now define $\tau:\ran(K)\rightarrow\ell^2(\B)$ by $\tau(z)=\{\langle z/y_i\rangle\}_{i\in\N},~\forall z\in\ran(K)$. We have $\tau(\flat(h,y_j))=\{\langle\flat(h,y_j)/y_i\rangle\}_{i\in\N}=\{\delta_{ij}h\}_{i\in\N}$. Now considere the operator $\Omega:\ran(K)\rightarrow\ran(K)$, such that $\Omega=T\tau$. It is easy to check that $\Omega$ is bounded and surjective and that
$$\Omega(\flat(h,y_i))=T(\{\delta_{ij}h\}_{i\in\N})=\flat(h,x_i), ~~~\forall h\in\h, \forall i\in\N.$$
\end{itemize} 
\end{proof}

\begin{thm}
Suppose that $\{x_i\}_{i\in\N}$ is a $\flat$-orthonormal basis for $\z$, and let $U:\X\rightarrow\X$ be a continuous bounded linear operator. If $\{x_i\}_{i\in\N}$ is a $\flat$-frame for $\z$ with bounds $A$, and $B$, then $\{Ux_i\}_{i\in\N}$ is a $(U^{\flat*})$-$\flat$-frame for $\z$ with bounds $A'=A$ and $B'=B\Vert U^{\flat}\Vert^2.$ Where $U^{\flat*}$ designs the adjoint operator of $U^\flat$ on $\z$.
\end{thm}

\begin{proof}
$\{x_i\}_{i\in\N}\subset\B$ is a $\flat$-frame for $\z$, then for all $z\in\z$,
$$A\Vert z\Vert^2_{\z}\leqslant\sum_{i=1}^{\infty}\Vert \langle z/x_i\rangle\Vert^2_{\h}\leqslant B\Vert z\Vert^2_{\z}.$$
We have by theorem \ref{exstadj} 
$$\sum_{i=1}^{\infty}\Vert \langle z/Ux_i\rangle\Vert^2_{\h}=\sum_{i=1}^{\infty}\Vert \langle U^{\flat} z/x_i\rangle\Vert^2_{\h}$$
hence \begin{equation}\label{bfandkf}
A\Vert U^z\Vert^2_{\z}\leqslant\sum_{i=1}^{\infty}\Vert \langle z/Ux_i\rangle\Vert^2_{\h}\leqslant B\Vert U^{\flat} z\Vert^2_{\z}.
\end{equation}
which implies that 
$$A\Vert (U^{\flat})^*z\Vert^2_{\z}\leqslant\sum_{i=1}^{\infty}\Vert \langle z/Ux_i\rangle\Vert^2_{\h}\leqslant B\Vert U^{\flat}\Vert\Vert z\Vert^2_{\z},~~\forall z\in\z.$$
\end{proof}

\begin{cor}
Suppose that $\{x_i\}_{i\in\N}$ is $\flat$-complete in $\z$, and let $U:\X\rightarrow\X$ be a continuous bounded linear operator. Let $U^{\flat}$ be the $\flat$-adjoint operator of $U$ in $\z$, and suppose that it has a closed range. If $\{x_i\}_{i\in\N}$ is a $\flat$-frame for $\z$, with bounds $A$, and $B$, then $\{Ux_i\}_{i\in\N}$ is also a $\flat$-frame for $\z$ with bounds $A'=\Vert U^{\flat\dagger}\Vert^{-2}$ and $B'=B\Vert U^{\flat}\Vert^2$. Where $U^{\flat\dagger}$ is the pseudo-inverse of $U^{\flat}$ on $\ran(U^{\flat})$.
\end{cor}

\begin{proof}
Suppose that $\ran(U^{\flat})$ is closed then $$\exists U^{\flat\dagger}\in\el(\z); ~~~ U^{\flat}U^{\flat\dagger}z=z,~~\forall z\in\ran(U^{\flat}),$$ so  $\Vert z\Vert_{\z} = \Vert U^{\flat}U^{\flat\dagger}z\Vert_{\z}$, then we have for all $z\in \ran(U^{\flat})$:
$$ \Vert z\Vert_{\z}^2 \leqslant \Vert U^{\flat}z\Vert_{\z}^2\Vert U^{{\flat}\dagger}z\Vert^2 \Rightarrow \Vert U^{\flat\dagger}z\Vert^{-2}\Vert z\Vert_{\z}^2 \leqslant \Vert U^{\flat}z\Vert_{\z}^2.$$
using (\ref{bfandkf}) we prove the result.
\end{proof}

\begin{thm}
Let $\{x_i\}_{i\in\N}\subset\B$ be a $\flat$-orthonormal basis in $\z$, and let $U:\X\rightarrow\X$ be a continuous bounded linear operator, and let $K\in\el(\z)$. Suppose that $\{x_i\}_{i\in\N}$ is a $K$-$\flat$-frame for $\z$, with bounds $A$, and $B$, Then $\{Ux_i\}_{i\in\N}$ is a $(U^{\flat*}K)$-$\flat$-frame for $\z$ with bounds $A'=A$ and $B'=B\Vert U^{\flat}\Vert^2.$ where $U^{\flat*}$ designs the adjoint operator of $U^{\flat}$ in $\z$.
\end{thm}

\begin{proof}
$\{x_i\}_{i\in\N}\subset\B$ is a $K$-$\flat$-frame for $\z$ with bounds $A$ and $B$, then 
$$A\Vert K^*z\Vert^2_{\z}\leqslant\sum_{i=1}^{\infty}\Vert \langle z/x_i\rangle\Vert^2_{\h}\leqslant B\Vert z\Vert^2_{\z}.$$
We have by theorem \ref{exstadj} 
$$\sum_{i=1}^{\infty}\Vert \langle z/Ux_i\rangle\Vert^2_{\h}=\sum_{i=1}^{\infty}\Vert \langle U^{\flat}z/x_i\rangle\Vert^2_{\h}$$
hence $$A\Vert K^*(U^{\flat}z)\Vert^2_{\z}\leqslant\sum_{i=1}^{\infty}\Vert \langle z/Ux_i\rangle\Vert^2_{\h}\leqslant B\Vert U^{\flat}z\Vert^2_{\z}.$$
which implies that 
$$A\Vert (U^{\flat*}K)^*z\Vert^2_{\z}\leqslant\sum_{i=1}^{\infty}\Vert \langle z/Ux_i\rangle\Vert^2_{\h}\leqslant B\Vert U^b\Vert\Vert z\Vert^2_{\z}.$$
\end{proof}

\newpage
\begin{minipage}[c]{\linewidth}
\begin{large}
\textbf{Declarations}
\end{large}
\end{minipage}

\begin{minipage}[c]{\linewidth}
\vspace{1cm}
\textbf{\textit{Disclosure statement:}} The authors declare that they didn't have any financial support, and
have no competing interests.
\end{minipage}

\begin{minipage}[c]{\linewidth}
\vspace{0.5cm}
\textbf{\textit{Acknowledgements:}} The authors would like to thank and express their sincere gratitude to the referee for their comments that will be helpful to improve this manuscript's quality.
\end{minipage}

\begin{minipage}[c]{\linewidth}
\vspace{0.5cm}
\textbf{\textit{Data availability statement:}} No data set associated with this paper.
\end{minipage}

\end{document}